\documentclass[letterpaper, 10 pt, conference]{ieeeconf}
\IEEEoverridecommandlockouts

\usepackage{cite}
\usepackage{float}
\usepackage{amsmath,amssymb,amsfonts}
\usepackage{mathrsfs}
\usepackage{graphicx}
\usepackage{textcomp}
\usepackage[ruled,vlined]{algorithm2e}
\usepackage{mathtools}
\usepackage{adjustbox}
\usepackage{makecell}
\usepackage{empheq} 
\usepackage{xcolor}
\usepackage[amsmath,thmmarks]{ntheorem}
\usepackage{hyperref}
\usepackage{enumerate}
\hypersetup{pdfborder={0 0 0},}
\def\BibTeX{{\rm B\kern-.05em{\sc i\kern-.025em b}\kern-.08em
    T\kern-.1667em\inner.7ex\hbox{E}\kern-.125emX}}
 \theoremheaderfont{\normalfont\bfseries} 
\theorembodyfont{\normalfont\itshape}    
\theoremseparator{.}                     

\newtheorem{remark}{Remark}
\newtheorem{theorem}{Theorem}
\newtheorem{lemma}{Lemma}

\newtheorem{assumption}{Assumption}

\begin{document}

\title{\LARGE \bf CADMM-Prox: A Bi-level Consensus ADMM for Non-smooth Non-convex Distributed Consensus Optimization
}

\author{Xu Du,  Shuting Wu,  Karl~H.~Johansson, and Apostolos I. Rikos$^*$
\thanks{$^*$Corresponding author.}
	\thanks{Xu Du and Apostolos I. Rikos are with the Artificial Intelligence Thrust of the Information Hub, The Hong Kong University of Science and Technology (Guangzhou), Guangzhou, China. 
    Apostolos I. Rikos is also affiliated with the Department of Computer Science and Engineering, The Hong Kong University of Science and Technology, Clear Water Bay, Hong Kong, China. E-mails: {\tt~\{michaelxudu, apostolosr\}@hkust-gz.edu.cn}. 
            }
            \thanks{Shuting Wu is with School of Mathematics and Statistics, North China University of Water Resources and Electric Power, Zhengzhou, China.  E-mail: \texttt{wushuting0126@163.com}.}
            \thanks{Karl H.~Johansson is with the Division of Decision and Control Systems, KTH Royal Institute of Technology, SE-100 44 Stockholm, Sweden. 
    He is also affiliated with Digital Futures, SE-100 44 Stockholm, Sweden. 
    E-mail:{\tt~kallej@kth.se}.
            }
            \thanks{The work of X.D. and A.I.R. was supported by the Guangzhou-HKUST(GZ) Joint Funding Scheme (Grant No. 2025A03J3960). The work of A.I.R. was also supported by the Guangdong Provincial Project (Grant No. 2024QN11G109).}
}

\maketitle

\begin{abstract}
Non-smooth and non-convex optimization problems are pervasive in machine learning, control, and signal processing, due to the need for sparse solutions and the inherently non-convex nature of many objective functions. 
In this paper, we study non-smooth and non-convex distributed optimization problems. 
We propose a novel bi-level Consensus Alternating Direction Method of Multipliers (ADMM) algorithm, termed CADMM-Prox. 
The proposed algorithm integrates classical Consensus ADMM with a proximal mechanism by introducing a sufficiently large proximal term associated with an outer-level variable. 
Under the mild assumption that the local objective functions are semi-convex, CADMM-Prox is guaranteed to converge globally to a neighborhood of a generalized stationary point. 
Numerical experiments on a phase retrieval problem demonstrate that our proposed method exhibits more stable convergence behavior compared with baseline algorithm. 
\end{abstract}

\section{Introduction}
Distributed optimization has attracted significant attention in recent years due to its wide range of applications in control systems \cite{stomberg2025decentralized}, power networks \cite{dai2026distributed}, and machine learning \cite{zhou2026preconditioned}. 
These applications demonstrate that, with the aid of distributed optimization, large-scale problems can be handled in a scalable manner while mitigating the effect of the \emph{curse of dimensionality} \cite{2024_doostmohammadian_rikos_Johansson_survey}.

Distributed optimization algorithms are designed to solve finite-sum problems in which local decision variables are either affinely coupled or required to reach consensus. 
In this paper, we focus on distributed consensus optimization.
A typical distributed optimization algorithm consists of two fundamental steps: a local minimization step and a coordination step \cite{2024_doostmohammadian_rikos_Johansson_survey,boyd2011distributed}. 
Coordination among agents can be achieved either with or without a centralized coordinator. 
When agents exchange information directly without relying on a central coordinator, the resulting scheme is referred to as decentralized optimization \cite{ling2015dlm}. 
In contrast, this paper considers the distributed setting assuming the existence of a centralized coordinator. 
Additionally, distributed optimization methods are broadly categorized into two main classes \cite{2024_doostmohammadian_rikos_Johansson_survey}: 
(i) primal decomposition methods, and 
(ii) dual decomposition methods. 
This work focuses on the latter class, particularly the alternating direction method of multipliers (ADMM) \cite{boyd2011distributed}.

\noindent
\textbf{Existing Literature.}
ADMM was originally introduced by \cite{Gabay1976} and \cite{glowinski1975approximation}, and later popularized by \cite{boyd2011distributed} for a wide range of applications. 
To the best of our knowledge, a tailored ADMM framework for distributed consensus optimization was first systematically presented in \cite[Chapter~7]{boyd2011distributed}. 
The convergence theory of ADMM has since developed significantly. 
A milestone result in \cite{he20121} established sublinear convergence of two-block Gauss–Seidel (serial)  ADMM under convexity assumptions without requiring smoothness. 
This was later extended to multi-block convex problems with Jacobi type (parallel) updates \cite{du2026affine}.  
When additional smoothness assumptions are imposed, linear convergence rates can be achieved \cite{ling2015dlm}. 
However, these works suffer from at least one of three limitations: (i) they are confined to convex problems, (ii) they rely on smoothness, or (iii) they do not support Jacobi type updates among agents.
To concurrently tackle the first and third limitations, convergence guarantees for Jacobi-type ADMM \cite{Hong2016, zhang2020proximal} and other first-order method \cite{doostmohammadian2024nonlinear} have been established. However, their convergence hinges on parameter choices (e.g., the penalty parameter for ADMM and the step-size for gradient methods) that must be based on the Lipschitz constants.
To simultaneously address the first and second limitations, several works have made preliminary progress 
   \cite{wang2019global, el2025complexity}. 
Nevertheless, their theoretical guarantees still require at least one smooth component in the objective function, since the convergence analysis relies on a sufficiently decreasing augmented Lagrangian. This requires the smooth component to have a Lipschitz‑continuous gradient, which provides a quadratic descent term to compensate for fluctuations introduced by the non-smooth and non‑convex parts. 
In addition, their analysis requires agents to update their local variables in a Gauss–Seidel manner. 
Moreover, to overcome the first two limitations (convexity and smoothness), \cite{zeng2022moreau_published,zeng2021moreau} established the global convergence theory of the augmented Lagrangian method (ALM) for centralized non-convex and non-smooth problems by analyzing a lower bound on the augmented Lagrangian parameter. 
The work in \cite{zeng2021moreau} also presented a numerical example of an ADMM type method, however, without providing a rigorous convergence guarantee. 
In Table~\ref{table: ADMM} we summarize several representative ADMM type methods and the corresponding assumptions for their convergence. 
Additionally, we refer interested readers to \cite{yang2022survey} for more details. 
From the aforementioned works, it is evident that existing studies address at most one or two of the three limitations, namely convexity, smoothness, and Gauss–Seidel  updates. To the best of our knowledge, no work has yet simultaneously overcome all three while providing rigorous global convergence guarantees. As a result, under Jacobi type updates among agents, the study of ADMM type algorithms for non-smooth and non-convex problems remains an open challenge. 
\vspace{-.1cm} 
\begin{table}[ht]
		\caption{Assumptions of ADMM type Algorithms} \vspace{-.3cm}
		\label{table: ADMM}
		\begin{center}
       \scalebox{0.9}{\begin{tabular}{|c|c|c|}
				\hline
				Attribute
				&\makecell{Smooth}& \makecell{Non-smooth}\\
				\hline
				\makecell{Convex} &\makecell{ \cite{ling2015dlm}, \cite{jang2026alia}\\ \cite{du2025decentralized,Du2025ecc}}  &\makecell{\cite{he20121},\cite{deng2017parallel,du2026affine},\\\cite{rikos2023asynchronous,jiang2021asynchronous,wei2012distributed}} \\
				\hline
				\makecell{Non-convex}  &\makecell{\cite{Hong2016,zhang2020proximal,sun2023two}} &\makecell{\cite{wang2019global,el2025complexity} (Gauss-Seidel update),\\
                \textbf{CADMM-Prox} (Jacobi type update)} \\
				\hline
			\end{tabular}} 
		\end{center}
	\end{table}
    
\noindent
\textbf{Main Contributions.} Motivated by the aforementioned limitations in the existing literature, we introduce a novel bi-level Consensus ADMM framework. 
To the best of our knowledge, this is the first Consensus ADMM type method that simultaneously handles (i) non-smoothness without requiring smooth surrogates, (ii) non-convexity, (iii) does not require damping techniques (e.g., line search, trust regions, or projections), and (iv) Jacobi type updates across agents. 
Our main contributions are summarized as follows. \\
\noindent
\textbf{A.} 
We propose CADMM-Prox, a bi-level Consensus ADMM algorithm that decouples the challenges of non-smoothness and non-convexity (see Algorithm~\ref{alg: Globalization of ADMM}). 
Under a mild semi-convexity assumption, we construct a convex surrogate of the original non-convex and non-smooth consensus problem by introducing a sufficiently large proximal term on an outer-level variable. The inner level retains the classical Consensus ADMM structure to solve this convex surrogate, while the outer-level variable is updated according to a prescribed criterion \eqref{eq: criterion}. 
We establish the global convergence of CADMM-Prox (see Theorem~\ref{them: global convergence}). 
Under the assumption that each local objective function is closed, proper, and semi-convex, the proposed algorithm is guaranteed to converge to a neighborhood of a generalized stationary point of the non-smooth and non-convex distributed consensus optimization problem. \\
\noindent
\textbf{B.} 
We evaluate the performance of our proposed algorithm in a phase retrieval application, analyzing its numerical convergence stability. Additionally, we benchmark it against state-of-the-art algorithms from the literature, even though the latter lack rigorous global convergence guarantees. 
 
 \vspace{.2cm}
\textit{Notation.} In this paper, $(\cdot)^{[k]}$ denotes the value at the $k$-th iteration of the algorithm. 
The notation $\xi^\top$ represents the transpose of a vector $\xi \in \mathbb{R}^n$. The symbol $\|a\|$ stands for the Euclidean norm of a vector $a \in \mathbb{R}^n$. Furthermore, $|\mathcal{S}|$ indicates the cardinality of a countable set $\mathcal{S}$.

\section{Problem Formulation}\label{sec: Problem Formulation and Preliminaries}

We now review the concept of 
$\alpha$-semi-convex functions, which underpins the algorithm design. 
Then, we formulate the non-smooth 
$\alpha$-semi-convex distributed consensus optimization problem and introduce a convex surrogate. 

\subsection{ 
The $\alpha$-Semi-Convex Function
}\label{sec: L2}

This paper focuses on a class of non-smooth and non-convex problems, namely \emph{
$\alpha$-semi-convex} (also referred to as weakly convex or hypo-convex) \cite[Definition~10, Eq.~(18)]{semiconvex}, \cite{bohm2021variable}.  
Specifically, we assume the objective function $f:\mathbb{R}^n \rightarrow \mathbb{R}$ is non-smooth and globally satisfies
\begin{equation}\label{eq:semi-convex}\small
f(\xi) + d^\top (\zeta - \xi)
\leq f(\zeta) + \frac{\gamma}{2} \left\| \xi - \zeta \right\|^2,
\end{equation}
for all $\xi, \zeta \in \mathbb{R}^n$ and any sub-gradient $d \in \partial f(\xi)$. Here, $\alpha>0$ is a finite constant, which can be associated with an $L_2$ proximal term to ensure \eqref{eq:semi-convex} holds globally.  
We note that the related concept of \emph{prox-regularity} \cite{wang2019global} constitutes a weaker assumption than 
$\alpha$-semi-convexity, since prox-regular functions require \eqref{eq:semi-convex} to hold only locally, whereas 
$\alpha$-semi-convex functions satisfy it globally.  
For an 
$\alpha$-semi-convex function $f$ satisfying \eqref{eq:semi-convex}, adding an $L_2$ proximal term associated with 
with a regularization parameter $\gamma>\alpha$ yields the function
\begin{equation}\label{eq: convex}\small
F^{z}(x) \doteq f(x) + \frac{\gamma}{2} \left\| x - z \right\|^2, 
\end{equation}
which is $(\gamma-\alpha)$-strongly convex with respect to $x$
for every $z\in\mathbb{R}^n$, and hence strictly convex.

\subsection{Distributed Consensus Optimization Problem}\label{sec: Problem Formulation}
Let $\mathcal{V}$ denote the set of agents, where $N = |\mathcal{V}|$ represents the number of agents. 
Distributed optimization can be formulated as
\begin{equation}\label{eq: DC0}\small
\begin{aligned}
   \min_{y \in \mathbb R^{n}} \quad &  \sum_{i=1}^{N} f_i(y), 
\end{aligned}
\end{equation}
where $f_i(\cdot): \mathbb R^n \rightarrow \mathbb R$ denotes the local objective function of agent $i \in \mathcal V$, which is non-smooth and 
$\alpha$-semi-convex (i.e., it satisfies \eqref{eq:semi-convex}). 
To enable distributed computation, we reformulate \eqref{eq: DC0} by introducing local variables $x_i$ for each agent $i \in \mathcal V$.
In this case, \eqref{eq: DC0} becomes:  
\begin{equation}\label{eq: DC}\small
    \begin{aligned}
   \textbf{P1:} \quad     \min_{x_i,\, i \in \mathcal V, y} \quad & \sum_{i=1}^{N} f_i(x_i) \qquad
        \text{s.t.}\;\;\;  x_i = y, \quad \forall i \in \mathcal V.
    \end{aligned}
\end{equation}
Problem~\textbf{P1} is referred to as the distributed consensus optimization problem, with the constraints enforcing agreement among all local variables. 
Directly solving the original problem in \eqref{eq: DC} is challenging due to its non-smooth and non-convex nature. 
To overcome these difficulties, we introduce a sufficiently large $L_2$ proximal term, which convexifies each non-convex component $f_i(\cdot)$ under the 
$\alpha$-semi-convexity assumption. 
Consequently, we now introduce the following strongly convex surrogate problem:
\begin{equation}\label{eq: Convex DC}\small
    \begin{aligned}
    \textbf{P2:} \quad    \min_{x_i,\, i \in \mathcal V, y} \quad 
        & \sum_{i=1}^{N} F_i^{z}(x_i) \qquad
        \text{s.t.}  \;\;
         x_i = y, \quad \forall i \in \mathcal V,
    \end{aligned}
\end{equation}
where $F_i^{z}(x_i) \doteq f_i(x_i) + \frac{\gamma}{2} \left\| x_i - z \right\|^2$ as defined in \eqref{eq: convex}.
The convex surrogate problem in \eqref{eq: Convex DC} plays a central role in the algorithmic design developed in this paper. 
For notational convenience in the subsequent analysis, let us now define
\begin{equation}\label{eq: global function}\small 
\Phi(z, y) \doteq \sum_{i=1}^{N} F_i^{z}(y) , 
\end{equation}
for which it holds,
\begin{equation}\label{eq: global function2}\small
\Phi(z, z) = \sum_{i=1}^{N} f_i(z).
\end{equation} 
Equation \eqref{eq: global function2} reveals the relationship between \textbf{P1} and \textbf{P2}. 
When $y = z$, the objective function values of the two problems coincide. 
This fact will be instrumental in establishing the subsequent global convergence results. 



\section{Preliminaries on Consensus ADMM} \label{sec: CADMM}

The augmented Lagrangian function of \textbf{P1} 
can be formulated as
\begin{equation}\label{eq: AL}\small
    \mathcal L (x, y, \lambda) 
    \doteq \sum_{i=1}^N\left( 
    f_i(x_i) 
    + \lambda_i^\top (x_i - y ) 
    + \frac{\rho }{2}\left\|x_i - y\right\|^2
    \right), 
\end{equation}
where, $x = [x_1^\top, x_2^\top, \dots, x_N^\top ]^\top$ 
collects the local primal variables, and $\lambda = [\lambda_1^\top, \lambda_2^\top, \dots, \lambda_N^\top]^\top$
collects the dual variables.
Based on \eqref{eq: AL}, the consensus ADMM method \cite[Chapter~7]{boyd2011distributed} is given by
\begin{subequations}\label{eq: CADMM}\small
\begin{align}
x_i^{[t+1]} 
&= \mathop{\arg\min}_{x_i, 
\forall i \in \mathcal V, } 
\; f_i(x_i) 
+ \left(\lambda_i^{[t]}\right)^\top x_i 
+ \frac{\rho}{2}\left\|x_i - y^{[t]}\right\|^2,
\label{eq: local update} \\
y^{[t+1]} 
&= \frac{1}{N}
\sum_{i=1}^N
\left(
x_i^{[t+1]} + \frac{\lambda_i^{[t]}}{\rho}
\right), 
\label{eq: z update} \\
\lambda_i^{[t+1]} 
&= \lambda_i^{[t]} 
+ \rho\left(x_i^{[t+1]} - y^{[t+1]}\right), 
\label{eq: dual update}
\end{align}
\end{subequations}
for every agent $i \in \mathcal V$, with $t$ denoting the iteration index.
The algorithm consists of three main steps. 
In \eqref{eq: local update}, each agent updates its local primal variable $x_i^{[t+1]}$. 
In \eqref{eq: z update}, the global variable $y^{[t+1]}$ is updated. 
Finally, the dual variables $\lambda_i$ are updated according to \eqref{eq: dual update}. 

As shown in \cite{boyd2011distributed, rikos2023asynchronous}, when each $f_i(\cdot)$ is convex, the iterates generated by \eqref{eq: CADMM} converge globally. 
In contrast, for non-smooth  non-convex problems, global convergence results are significantly less developed and remain largely limited in the literature. 
Our proposed algorithm in this paper enables the solution of \eqref{eq: DC} when the objective functions are (i) non-smooth, (ii) 
$\alpha$-semi-convex, (iii) do not require damping techniques, and (iv) are updated across agents in a Jacobi type manner, all while guaranteeing global convergence without increasing the per-iteration computational complexity.

\section{CADMM-Prox}\label{sec: ALgorithm}

In this section, we introduce a novel distributed algorithm, for solving the non‑smooth and non‑convex problem \textbf{P1} in Section~\ref{sec: Problem Formulation}. 
Our algorithm is described below as Algorithm~\ref{alg: Globalization of ADMM} (CADMM-Prox). 
Before presenting our proposed algorithm, we make the following assumption which is fundamental for the subsequent development.



\begin{assumption}\label{ass} 
The local cost function $f_i$ of each agent $i\in\mathcal V$ is closed,
proper, $\alpha$-semi-convex, and bounded from below. Specifically, there
exists a finite constant $\alpha>0$ such that \eqref{eq:semi-convex} holds for each
$f_i(\cdot)$. Moreover, the regularization parameter $\gamma$ is chosen
such that $\alpha<\gamma<\infty$.
\end{assumption}

Assumption~\ref{ass} is necessary for designing our algorithm and for establishing its global convergence. 
The 
$\alpha$-semi-convexity ensures a finite 
$\gamma$ for constructing the strongly convex surrogate \eqref{eq: Convex DC}. 
As a consequence, the augmented problem in \eqref{eq: Convex DC} is guaranteed to be feasible, to satisfy strong duality, and can be solved by the inner level of Algorithm \ref{alg: Globalization of ADMM}.
Furthermore, boundedness from below supports the convergence analysis and guarantees finite-time termination.

\subsection{Algorithm Development}\label{sec: CADMM-Prox}

In this section, we present our proposed decentralized
algorithm, detailed below as Algorithm~\ref{alg: Globalization of ADMM} (CADMM-Prox).
			\begin{algorithm}[ht]\small
				\caption{CADMM-Prox: Consensus ADMM with Proximal Term}
				\label{alg: Globalization of ADMM}
				\textbf{Initialization:} Outer and inner level global variables $z^1, y^1\in \mathbb R^n$ and the dual $\lambda_i^1\in \mathbb R^n$. Set $
                \gamma>\alpha$ $\rho>0,$ $\beta\geq0$, $\epsilon>0$, $t=1$, $k=1$. \\
				\Repeat{  $\left\|{z}^{[k+1]} - {z}^{[k]}\right\| < \epsilon$ }{{(\textit{Outer Level})}\\
					$k \leftarrow k+1$\;
					\Repeat{   The following inequality is satisfied    }{(\textit{Inner Level})\\
						$t \leftarrow t+1$\;\begin{enumerate}
							\item 
							Update the local variable as $x_i^{[t+1]}$ with 
							\begin{equation*}\label{eq: local update ADMM}\small
								\hspace{-1.9cm}\begin{split}
					{x_i^{[t+1]}}=\mathop{\arg\min}_{x_i, i\in \mathcal V} &\hspace{0.1cm} F^{ z^{[k]}}_i(x_i)+\left(\lambda_i^{[t]}\right)^\top  x_i\\
			&+\frac{\rho}{2}\left\|x_i-y^{[t]}\right\|^2.
								\end{split}
							\end{equation*}
							
							\item  Update the inner level global variable $y$ with \\equation \eqref{eq: z update}.
							\vspace{0.1cm}
	\item Update the dual variables $\lambda_i$s with \eqref{eq: dual update}.

					\end{enumerate}}
                    \begin{equation}\label{eq: criterion}\small
\Phi\left(z^{[k]}, y^{[t+1]}\right) < \Phi\left(z^{[k]}, z^{[k]}\right).
\end{equation}
					\vspace{0.1cm}
                    4)\hspace{0.06cm} Set $ z^{[k+1]} = y^{[t+1]} $.
					\vspace{0.2cm}}
				\vspace{0.1cm}
				
				\textbf{Output:} Approximate optimal solution $ z^* $ of problem  \eqref{eq: DC}. 
			\end{algorithm}

Our proposed CADMM‑Prox (Algorithm~\ref{alg: Globalization of ADMM}) is a bi‑level distributed algorithm that solves the non‑smooth non‑convex problem~\textbf{P1} in \eqref{eq: DC} while enabling Jacobi‑type updates among agents. 
The bi‑level structure decouples the challenges of non‑smoothness and non‑convexity and consists of two layers, an inner, an outer. 
The inner level (steps 1–3) solves the non‑smooth convex surrogate problem \textbf{P2} in \eqref{eq: Convex DC} parameterized by the current outer level variable $z^{[k]}$. 
When the given criterion in \eqref{eq: criterion} is satisfied, the outer level (step 4) updates the outer variable as $z^{[k+1]} = y^{[t+1]}$.  
This point further decreases the merit function in \eqref{eq: global function2} and is used to construct the convex surrogate problem \textbf{P2} for the next inner iteration. The algorithm alternates between solving the convex surrogate \textbf{P2} at the inner level and updating $z^{[k]}$ at the outer level until the outer convergence condition $\left\|z^{[k+1]}-z^{[k]}\right\|<\epsilon$ is met. Through this alternating process, the convex surrogate \textbf{P2} effectively converges to the original problem \textbf{P1}.


In the inner level of Algorithm~\ref{alg: Globalization of ADMM}, the convexity of the distributed consensus problem \textbf{P2} in \eqref{eq: Convex DC} is guaranteed by Assumption~\ref{ass}. The problem can be solved using classical Consensus ADMM \eqref{eq: CADMM} \cite{wei2012distributed, khatana2020d}, which is guaranteed to converge globally.
The main difference between the inner level of Algorithm~\ref{alg: Globalization of ADMM} and classical Consensus ADMM \eqref{eq: CADMM} lies in the construction of the local augmented Lagrangian, which employs $F_i^{z^{[k]}}(\cdot)$ instead of $f_i(\cdot)$ (see step 1).
The global variable update in step 2) and the dual variable update in step 3) coincide with their counterparts in classical Consensus ADMM, given by \eqref{eq: z update} and \eqref{eq: dual update}, respectively, and are omitted here for brevity.

\begin{remark}
To avoid the prohibitive cost of an exact inner-level solve, we adopt a practical stopping criterion \eqref{eq: criterion}.
This mechanism ensures a strict decrease of the merit function $\Phi(\cdot,\cdot)$ in \eqref{eq: global function2} with respect to the outer-level variable $z$, thereby guaranteeing global convergence of Algorithm~\ref{alg: Globalization of ADMM}.
We note that Algorithm~\ref{alg: Globalization of ADMM} is likewise applicable to the remaining problem classes summarized in Table~\ref{table: ADMM}, but the corresponding analysis is omitted here for conciseness. 
\end{remark}


\noindent
\textbf{Comparison with Previous Works.} In terms of problem types and solution strategies for the distributed consensus optimization problem \eqref{eq: DC}, existing ADMM type methods only partially address this setting. 
For convex and non-smooth problems, convergence of distributed ADMM has been established in \cite{boyd2011distributed, rikos2023asynchronous, jiang2021asynchronous}. 
For smooth but non-convex objectives, convergence results can be found in \cite{Hong2016, sun2023two}. 
The works closest to our setting are \cite{wang2019global, el2025complexity}, which consider non-smooth and non-convex problems. 
However, these methods rely on Gauss–Seidel type updates and additionally require smooth components in the objective functions. 
Consequently, although the aforementioned works satisfy part of our requirements, none of them simultaneously accommodate non-smoothness, non-convexity, and Jacobi type distributed updates with global convergence guarantees. 
Regarding computational complexity, the inner level of Algorithm \ref{alg: Globalization of ADMM} does not increase the per-iteration cost compared to classical Consensus ADMM. 
This is because the local update step \eqref{eq: local update} already includes an augmented Lagrangian term of the same structure, and the added proximal term does not change the order of complexity. 
Furthermore, the outer level only performs an assignment update of the global variable without solving an optimization problem, making its computational cost negligible. 
In contrast, the bi-level ADMM proposed in \cite{sun2023two} requires optimization at both levels, resulting in significantly higher complexity than single-level ADMM methods \cite[IV.E.1]{yang2022survey}. 

\subsection{Convergence Analysis}\label{sec: Convergence}

In this section, we present the global convergence analysis of Algorithm \ref{alg: Globalization of ADMM}. 
We first introduce two lemmas that serve as the foundation for the subsequent convergence results.

\begin{lemma}\label{lemma: convex convergence} 
Let Assumption~\ref{ass} hold. Then, for every $z^{[k]}$, each local
function $F_i^{z^{[k]}}(\cdot)$ in \eqref{eq: Convex DC} is closed, proper, and
$(\gamma-\alpha)$-strongly convex, and hence strictly convex.
Consequently, problem \eqref{eq: Convex DC} is feasible and satisfies strong duality.
Under these conditions, the inner-level iterations of Algorithm~\ref{alg: Globalization of ADMM}
applied to \eqref{eq: Convex DC} converge globally with an $ O(\frac{1}{t})$ sublinear rate.
\end{lemma}

\emph{Proof.} 
Since each $f_i$ is $\alpha$-semi-convex and $\gamma>\alpha$,
the regularized function $F_i^{z^{[k]}}(\cdot)$ is
$(\gamma-\alpha)$-strongly convex for every $z^{[k]}$, and hence
strictly convex. Since the inner-level subroutine of Algorithm~\ref{alg: Globalization of ADMM} coincides with the classical Consensus ADMM scheme in \eqref{eq: CADMM}, its globally sublinear convergence follows directly from the standard theory for convex consensus problems \cite{wei2012distributed, khatana2020d}.
\hfill$\blacksquare$

Lemma~\ref{lemma: convex convergence} provides the convergence guarantee for the inner-level iterations. The inner loop is terminated by the stopping criterion in \eqref{eq: criterion}. This criterion is triggered before the convex surrogate problem converges exactly. The $O(1/t)$ convergence rate ensures that the stopping criterion is satisfied within a finite number of iterations. The following lemma quantifies the corresponding decrease of the global merit function.


\begin{lemma}\label{lemma: descent lemma}
Suppose that Assumption~\ref{ass} holds for problem~\eqref{eq: DC}, and that the conditions of Lemma~\ref{lemma: convex convergence} are satisfied. 
Let the stopping criterion \eqref{eq: criterion} be triggered at iteration $(t+1)$ of the inner level of Algorithm~\ref{alg: Globalization of ADMM}. 
Then, the following inequality holds:
\begin{equation}\label{eq: descent}\small
    \Phi\left(y^{[t+1]},y^{[t+1]}\right) < \Phi\left(z^{[k]},z^{[k]}\right) - \frac{\gamma N}{2} \left\|y^{[t+1]} - z^{[k]}\right\|^2.
\end{equation}
\end{lemma}
\emph{Proof.} See Appendix \ref{app: descent}.
 \hfill$\blacksquare$


Lemma~\ref{lemma: descent lemma} quantifies the improvement achieved after the inner level of Algorithm \ref{alg: Globalization of ADMM} terminates. 
This result also reveals, from another perspective, that even for convex problems, ADMM does not necessarily guarantee monotonic convergence of the objective value. 
Such non-monotonic behavior has been numerically illustrated in 
\cite{boyd2011distributed},
\cite{giesen2019combining,wang2023maximally},  
and \cite{dai2026distributed}.
From another perspective, \cite[Example~2.1]{Houska2016} shows that for non-convex problems, enforcing local strong convexity of subproblems via an augmented Lagrangian term alone does not guarantee global convergence of ALM type methods. 


Now we are ready to present the global convergence theory of Algorithm~\ref{alg: Globalization of ADMM} (CADMM-Prox) via the following theorem.

\begin{theorem}\label{them: global convergence}
Let Assumption \ref{ass} hold. 
Then, the iterates of the outer-level variable $z^{[k]}$ generated by Algorithm~\ref{alg: Globalization of ADMM} (CADMM-Prox) satisfy
\begin{equation}\label{eq: convergence}\small
   \lim_{k \to \infty} \left\|z^{[k+1]} - z^{[k]}\right\| = 0.
\end{equation}
\end{theorem}
\emph{Proof.} See Appendix \ref{app: convergence}.
 \hfill$\blacksquare$


\begin{remark}
Theorem~\ref{them: global convergence} establishes that Algorithm~\ref{alg: Globalization of ADMM} admits a limit point \(z^*\). This limit point lies in a neighborhood of a generalized stationary point of problem~\eqref{eq: DC}. This result is jointly determined by the inner stopping criterion \eqref{eq: criterion} and the outer termination condition \(\left\|z^{[k+1]} - z^{[k]}\right\| < \epsilon\).
Given the limiting outer variable \(z^*\), the inner-level variables \(x_i\) and \(y\) converge under the associated convex approximation, and thus remain within a neighborhood of a generalized stationary point of \textbf{P1} \cite{bolte2014proximal}. This indicates that, at convergence, the convex surrogate \textbf{P2} has effectively converged to the original problem \textbf{P1}.
\end{remark}

\section{Numerical Experiments on the Phase Retrieval Problem}\label{sec: numerical}
In this section, we illustrate the numerical performance of our Algorithm~\ref{alg: Globalization of ADMM} (CADMM-Prox) 
by applying it to the phase retrieval problem \cite{davis2019stochastic}. 
In the phase retrieval problem, a set of agents collaborate to recover a common signal from the magnitude of its linear measurements. This problem, which finds broad applications in imaging, optics, and signal processing, is verified to satisfy Assumption~\ref{ass}, see \cite{davis2019stochastic}.

For the phase retrieval problem, focusing on \textbf{P1}, the local objective function of each agent $i$ is defined as
\begin{equation}\label{eq: phase retrieval}\small
f_i(x_i) = \left| \left(a_i^\top x_i\right)^2 - b_i \right|.
\end{equation}
In our implementation, we set $N = 50$ agents and $n = 5$ decision variables per agent. The vectors $a_i \in \mathbb{R}^5$ are drawn independently and identically distributed (i.i.d.) from a normal distribution. The noise terms $w_i$ follow the same distribution, and the scalars $b_i = (a_i^\top y^*)^2 + w_i \in \mathbb{R}_+$ are constructed accordingly, where $y^*$ denotes a preset ground truth signal. 
This setup simulates a generic distributed phase retrieval instance. 
For further details, we refer interested readers to \cite{zeng2021moreau}.
It is known that problem \eqref{eq: phase retrieval} is weakly convex (i.e., semi-convex). 
To guarantee the convexity of the inner-level distributed optimization problem \eqref{eq: Convex DC}, 
we set $\gamma = 3 \max_i \|a_i\|^2$
in Algorithm~\ref{alg: Globalization of ADMM} (CADMM-Prox). 
The augmented Lagrangian parameter is chosen as $\rho = 10^3$.

\begin{figure}[ht]
\centering
\includegraphics[width=0.4\textwidth,height=0.25\textheight]{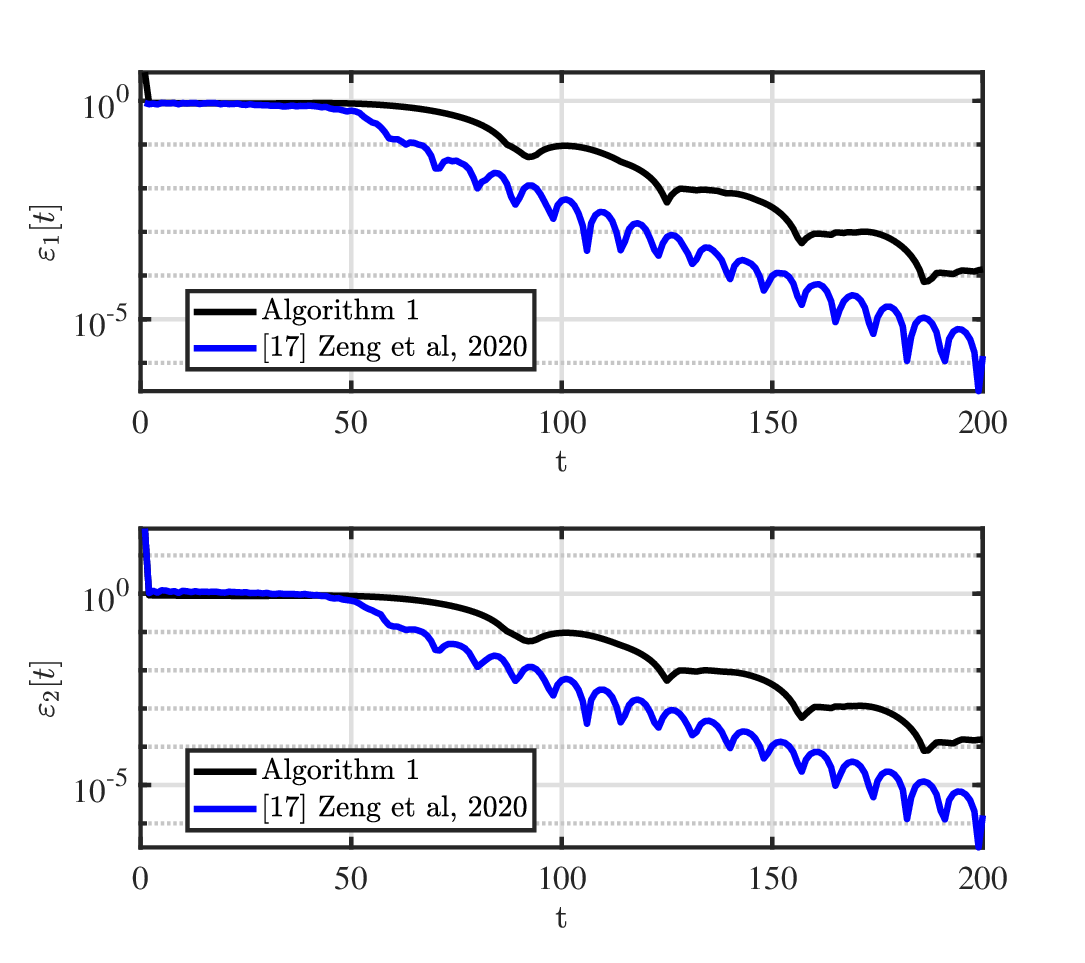}
\caption{Comparison of Algorithm~\ref{alg: Globalization of ADMM} and Consensus ADMM \cite[Section~7.3]{zeng2021moreau} on the phase retrieval problem \eqref{eq: phase retrieval}.}
\label{fig: comparison}
\end{figure}

In Fig.~\ref{fig: comparison} we compare the proposed Algorithm~\ref{alg: Globalization of ADMM} (CADMM-Prox) with the Jacobi type Consensus ADMM algorithm in \cite[Section~7.3]{zeng2021moreau}. 
We evaluate the algorithms using two performance metrics:
\begin{equation}\small
  \varepsilon_1[t] = \left \|y^{[t]} - y^*\right\|_\infty,\quad \varepsilon_2[t] = \frac{1}{N}\sum_{i=1}^N \left\|x_i^{[t]} - y^*\right\|_\infty.
\end{equation}
It is important to note here that to the best of the authors' knowledge, \cite[Section~7.3]{zeng2021moreau} provides a rare numerical example that directly applies Jacobi type updates across agents to solve a non-smooth, non-convex problem, without employing any smoothing or approximation techniques, and without relying on any damping mechanisms such as line search, diminishing step size or projections.
This setting is fully consistent with the structural requirements considered in this paper. 
However, as pointed out in \cite{wang2019global}, such an approach does not admit a theoretical convergence guarantee. 
We observe that although the method in \cite[Section~7.3]{zeng2021moreau} exhibits faster numerical convergence in the absence of theoretical guarantees, it also shows frequent oscillatory behavior during the iterations. 
In contrast, Algorithm~\ref{alg: Globalization of ADMM} (CADMM-Prox) demonstrates significantly reduced oscillations while being supported by rigorous global convergence theory.  
Reduced oscillations are critical for practical deployment, as they facilitate a more stable convergence process and lead to more reliable numerical solutions.

\section{Conclusion}

This paper proposed CADMM-Prox, a novel bi-level Consensus ADMM framework for solving non-smooth  non-convex distributed consensus optimization problems under Jacobi type updates. The bi-level structure decouples the challenges of non-smoothness and non-convexity. Under the 
$\alpha$-semi-convexity assumption, the inner-level problem is rendered convex by adding a sufficiently large proximal term with respect to the outer variable. As the outer variable is updated, this proximal term gradually vanishes, thereby recovering the original problem. This construction allows global convergence to be established via classical Consensus ADMM theory. 
Numerical experiments on a distributed phase retrieval problem illustrate stable convergence and competitive performance. 
Future work will focus on extending the proposed framework to distributed resource allocation problems and developing fully decentralized implementations over networked multi-agent systems.



\appendices
   \section{Proof of Lemma \ref{lemma: descent lemma} }\label{app: descent}
From the definition in \eqref{eq: global function}, when the stopping criterion \eqref{eq: criterion} is satisfied, we have
\begin{equation}\label{eq: monotone app1}\small
    \begin{split}
        \Phi\left(z^{[k]},z^{[k]}\right) \overset{\eqref{eq: criterion}}{>}& \Phi\left(z^{[k]}, y^{[t+1]}\right)
\\
\overset{\eqref{eq: global function},\eqref{eq: convex}}{=}&  \sum_{i=1}^N\left( f_i\left(y^{[t+1]}\right) + \frac{\gamma}{2}\left\| y^{[t+1]}- z^{[k]}\right\|^2 \right). \\
    \end{split}
\end{equation}
By adding the zero term $\frac{\gamma}{2}\left\| y^{[t+1]}- y^{[t+1]}\right\|^2$ to \eqref{eq: monotone app1}, we obtain
\begin{equation}\small
    \begin{split}
& \sum_{i=1}^N \left(f_i\left(y^{[t+1]}\right) + \frac{\gamma}{2}\left\| y^{[t+1]}- y^{[t+1]}\right\|^2\right)\\
&+ \frac{\gamma N}{2}\left\| y^{[t+1]}- z^{[k]}\right\|^2\\
\overset{\eqref{eq: global function2}}{=}& \Phi\left(y^{[t+1]},y^{[t+1]}\right) + \frac{\gamma N}{2}\left\| y^{[t+1]}- z^{[k]}\right\|^2. 
    \end{split}
\end{equation}
Inequality \eqref{eq: descent} then follows.

\section{Proof of Theorem \ref{them: global convergence}} \label{app: convergence}

From the outer-level update of Algorithm \ref{alg: Globalization of ADMM}, we have $z^{[k+1]} = y^{[t+1]}.$
From Lemma \ref{lemma: descent lemma}, it follows that
\begin{equation}\label{eq: descent_step}\small
\begin{split}
    &\left\|z^{[k+1]} - z^{[k]}\right\|^2 
    \overset{\eqref{eq: descent}}{<} \frac{2}{\gamma N} \left(\Phi\left(z^{[k]},z^{[k]}\right) - \Phi\left(z^{[k+1]},z^{[k+1]}\right)\right).
\end{split}
\end{equation}
Summing over outer-level iterations from $k=1$ to infinite, we obtain
\begin{equation}\label{eq: descent_sum}\small
\begin{split}
    &\sum_{k=1}^{\infty} \left\|z^{[k+1]} - z^{[k]}\right\|^2 \\
    \overset{\eqref{eq: descent_step}}{<} &\lim_{K\to \infty }\frac{2} {\gamma N} \Big(\Phi(z^{[1]},z^{[1]}) - \Phi(z^{[K+1]},z^{[K+1]})\Big).
\end{split}
\end{equation}
Since \(\Phi(\cdot,\cdot)\) is bounded below by Assumption \ref{ass}, and both \(\Phi(z^{[1]},z^{[1]})\) and \(\lim_{K \to \infty} \Phi(z^{[K+1]},z^{[K+1]})\) are finite, the right-hand side of \eqref{eq: descent_sum} is finite. Therefore, \(\sum_{k=1}^{\infty} \left\|z^{[k+1]} - z^{[k]}\right\|^2\) converges, which implies equation \eqref{eq: convergence}. This establishes the global convergence of the outer-level iterates.

\bibliographystyle{IEEEtran}   
\bibliography{references}      

\end{document}